\documentclass[11pt,a4paper,reqno]{amsart}

\usepackage{mathrsfs} 
\usepackage{hyperref} 

\usepackage{amsmath,amssymb,graphics,epsfig,color,enumerate,psfrag}
\usepackage{dsfont}  
\usepackage{verbatim}
\usepackage[normalem]{ulem} 
\newtheorem{Theorem}{Theorem}

\newtheorem{Lemma}{Lemma}
\newtheorem{Proposition}[Theorem]{Proposition}
\newtheorem{Corollary}[Theorem]{Corollary}
\newtheorem{Conjecture}{Conjecture}
\newtheorem{Remark}{Remark}

\theoremstyle{definition}
\newtheorem{Definition}{Definition}

\newcommand{\cF}{\ensuremath{\mathcal F}}

\newcommand{\cS}{\ensuremath{\mathcal S}}

\newcommand{\bbE}{{\ensuremath{\mathbb E}} }

\newcommand{\bbN}{{\ensuremath{\mathbb N}} }

\newcommand{\bbP}{{\ensuremath{\mathbb P}} }
\newcommand{\bbQ}{{\ensuremath{\mathbb Q}} }
\newcommand{\bbR}{{\ensuremath{\mathbb R}} }

\newcommand{\bbZ}{{\ensuremath{\mathbb Z}} }
\newcommand{\bfP}{{\ensuremath{\mathbf P}} }
\newcommand{\bfE}{{\ensuremath{\mathbf E}} }

    \let\d=\delta  \let\e=\varepsilon
 \let\g=\gamma       \let\l=\lambda
      \let\o=\omega      
      
  \let\z=\zeta
\let\D=\Delta

\newcommand{\one}{\mathds{1}}
\newcommand{\ARW}{\mathrm{ARW}}
\newcommand{\zagg}{\zeta_a} 
\newcommand{\zsta}{\zeta_s} 

\author[L.\ Levine]{Lionel Levine}
\address{Lionel Levine.
Department of Mathematics, Cornell University, Ithaca, NY 14853.
}
\email{levine@math.cornell.edu}

\author[V. Silvestri]{Vittoria Silvestri}
\address{Vittoria Silvestri. 
University of Rome La Sapienza, P.le Aldo Moro 5, 00185, Rome, Italy.}
\email{silvestri@mat.uniroma1.it}

\thanks{LL was partially supported by a Simons Fellowship, IAS Von Neumann Felowship, and NSF grant \href{https://www.nsf.gov/awardsearch/showAward?AWD_ID=1455272}{DMS-1455272}.}

\title{How far do activated random walkers spread from a single source?}

\begin{document}
\maketitle

\begin{abstract}
Unlike many particle systems, Activated Random Walk has nontrivial behavior even in one spatial dimension. We prove inner and outer bounds on the spread of $n$ activated random walkers from a single source in $\mathbb{Z}$.
The inner bound involves a comparison with the stationary distribution of activated random walkers on a finite interval, while the outer bound involves a comparison with the stabilization of an infinite Bernoulli configuration of activated random walkers on $\mathbb{Z}$.
\end{abstract}


\section{Three experiments, one theorem, and two conjectures}

Activated Random Walk is the name of a particle system with two species, 
 active particles ``A'' and sleeping particles ``S'' that become active when an active particle encounters them (``$A+S \rightarrow 2A$'').  Each active particle performs a continuous time random walk on the infinite path $\bbZ$, stepping to a random neighbour at rate 1. When an active particle is alone, it falls asleep (``A $\rightarrow$ S'') at rate $\lambda$. A sleeping particle stays asleep forever, until and unless an active particle steps to its location.  The parameter $\lambda>0$ is called the \emph{sleep rate}. Skip ahead to Section \ref{s.defs} for a precise definition of the model. 

The first rigorous results about Activated Random Walk were proved by Hoffman and Sidoravicius (unpublished) in the case of totally asymmetric walks, and by Rolla and Sidoravicius \cite{rolla2012absorbing} in the case of symmetric walks; we refer to \cite{rolla2020} for a complete history. Surprisingly, many basic questions about this simple one-dimensional particle system remain open!  In this paper we relate the outcomes of three different experiments.
\medskip

\textbf{Experiment 1.} For $\ARW(\bbZ, \lambda)$ (Activated Random Walk on $\bbZ$ with sleep rate $\lambda$), start with \emph{any} stationary ergodic configuration $\eta : \bbZ \to \bbN$, all initially active.  If each site of $\bbZ$ is visited only finitely often, then we say that $\eta$ \emph{stabilizes}; otherwise, we say that $\eta$ \emph{explodes}.
Rolla, Sidoravicius and Zindy \cite{rolla2019universality} proved the remarkable fact that stabilization depends only on the mean number of particles per site, $\zeta := \bbE ( \eta(0) )$. Namely, 
	there is a constant $\z_{c}(\bbZ, \lambda)$ such that if $\zeta < \z_{c}$ then
	$\eta$ stabilizes with probability $1$, and if $\zeta > \z_{c}$ then $\eta$ explodes with probability $1$.
This theorem applies to infinite configurations $\eta$ on the infinite path $\bbZ$. The next experiment concerns a finite configuration.
\medskip

\textbf{Experiment 2.} For $\ARW(\bbZ, \lambda)$, start $n$ walkers at $0$. How far do they spread? 

\begin{Conjecture}[Aggregate density $\zagg$] \label{c.zagg}
	There is a constant $\zagg(\bbZ,\lambda)$ such that for any $\e>0$, with probability tending to 1 as $n \to \infty$, the set of visited sites contains a centered interval of length $(1-\e)n/\zagg$, and is contained in an interval of length $(1+\e)n/\zagg$.
	\end{Conjecture}
	
\textbf{Experiment 3.} Fix a finite interval $I \subset \bbZ$ containing $0$, and write $\ARW(I, \lambda)$ for the finite particle system in which particles exiting $I$ are killed.  Start with one active particle at each site in $I$, and let them perform $\ARW(I, \lambda)$ until no active particles remain. How many of them survive? Write $|\cS_I(\one_I)|$ for the number of sleeping particles in $I$ at the end of this process.  

\begin{Conjecture}[Stationary density $\zsta$]
	There is a constant $\zsta(\bbZ,\lambda)$ such that 
		\[ \lim_{\# I \to \infty} \frac{|\cS_I(\one_I)|}{\# I} = \zsta \]
	in probability.
\end{Conjecture}

To explain why we call this the ``stationary'' density, consider the law $\pi_I$ of the random configuration of sleepers, $\cS_I(\one_I) : I \to \{0,s\}$.  The probability distribution $\pi_I$ is more canonical than it might seem. It is the unique stationary distribution of the Markov chain with state space $\{0,s\}^{I}$, whose update rule is: add one active particle at $v$ and stabilize. This distribution does not depend on the choice of $v \in I$. These facts are proved in \cite{levineliang}.

What is the relationship between the three densities $\zagg, \z_{c}, \zsta$?  It is tempting to speculate that $\zagg = \z_{c} = \zsta$. These and more detailed conjectures can be found in our companion paper \cite{LS2020}.  

 In the present paper we will prove inequalities of the form
\begin{equation}\label{zzz}
  \z_{out} \leq \zagg \leq \z_{in} 
 \end{equation}
where $\z_{out}$ and $\z_{in}$ are certain variants of the critical and stationary densities, respectively. 
To do this, we will prove inner and outer bounds on the aggregate $\cS (n\delta_0)$. The inner bound involves a comparison with the stationary configurations $\cS_I(\one_I)$, while the outer bound involves comparison with the infinite configurations $\cS(\eta)$ where $\eta$ is an i.i.d.\ Bernoulli configuration of mean $<\z_c$. 

\section{Main result}

We now define the outer and inner densities needed to state our main result. 

Fix $\z \in (0,1)$, and let $(\eta(x))_{x \in \bbZ}$ be independent random variables with $P(\eta(x)=1)=1-P(\eta(x)=0)=\z$. We interpret $\eta$ as a particle configuration consisting of one active particle each $x \in \bbZ$ such that $\eta(x)=1$. Denote by $w : \bbZ \to \bbN \cup \{\infty\}$ the $\ARW^\lambda$ odometer of $\eta$. This is the random function
	\[ w(x) = \text{number of $\ARW^\lambda$ firings needed at $x$ to stabilize $\eta$}. \]
Finally, for an interval $I = [a,b] \subset \bbZ $ write $\partial I =\{ a-1 , b+1 \}$ for the outer boundary of $I$. 

\begin{Definition}[Outer density]\label{def:zout}
For  $w$ as above, let 
	\[ \z_{out} = \z_{out} (\bbZ ,\l ) := \sup \{\z >0 : \bfE_\z  ( w (0)^3)<\infty \} . \]
\end{Definition}

Fix a sleep rate $0 < \lambda \leq \infty$, and an interval $I \subset \bbZ$. Write $\cS_I(\one_I)$ for the $\ARW^\lambda$ stabilization of the all-active configuration $\one_I$ with sink at the (outer) boundary of $I$. This is a random configuration of sleeping particles in $I$. Write $|\cS_I(\one_I)|$ for the total number of particles in this configuration.

\begin{Definition}[Inner density]\label{def:zin}
For any interval $I \subseteq \bbZ$ we let 
	\[ \z_{in, I} = \z_{in, I}(\bbZ , \l ) := \inf \left\{  \z >0 : \bbP (|\cS_I(\one_I)| > \z \# I ) \leq ( \# I)^{-20} \right\} , \]
and define 
	\[ \z_{in} := \limsup_I \z_{in,I}  . \]
\end{Definition}

Our main result addresses the question in the title: How far do activated random walkers spread from a single source? Starting with $n$ active particles at $0$, let $A_n$ be the random set of sites in $\bbZ$ that fire at least once during $\ARW^\l$ stabilization.  Write $B_r = [-r,r] \cap \bbZ$.

\begin{Theorem}\label{th:intro}
Let   $\z_{out} = \z (\bbZ, \l )$ and $\z_{in} = \z_{in}(\bbZ, \l )$ be as in Definitions \ref{def:zout} and \ref{def:zin} respectively, and assume $ \z_{out} >0$. Then for all $\varepsilon>0$ it holds 
	\[ \bbP \left( 
	\# A_n \geq \frac{n}{\z_{in} + \e } 
	\mbox{ and } 
	A_n \subseteq  B_{\frac{n}{2\z_{out}} (1+\varepsilon )} \mbox{ eventually in }n \right) =1.\]
\end{Theorem}
Note that $A_n$ is an interval containing the origin, so the above result tells us that $A_n$ contains an interval (possibly not centered!)\ of length $ \frac{n}{\z_{in} + \e }$ and is contained in a (centered) interval of length $\frac{n}{\z_{out}} (1+\varepsilon )$. As a byproduct, we obtain the following inequality.

\begin{Corollary}
$ \z_{in} \geq \z_{out}  $. 
\end{Corollary}
\begin{proof}
By Theorem \ref{th:intro} we have that almost surely for all $\e \in \bbQ _{>0}$ it holds 
	\[ \frac n{\z_{in} + \e } \leq \# A_n \leq \frac{n}{\z_{out} }(1+\e ) \]
eventually in $n$. This forces 
	\[ \z_{in} + \e  \geq \frac{\z_{out} }{1+\e } , \]
and sending $\e \to 0$ along rationals gives the desired conclusion. 
\end{proof}

In fact we conjecture equality.

\begin{Conjecture}
$\z_{in}=\z_{out}$.
\end{Conjecture}

Together with Theorem~\ref{th:intro}, this conjecture would imply that the random set $A_n$ is asymptotically a centered interval, in that
	\[ \bbP \left( B_{(1-\e)r} \subset A_n \subset B_{(1+\e)r} \text{ eventually} \right) = 1 \]
where $r = \frac{n}{2\z_{out}}$.

\begin{Remark}
For any $\gamma >0$ we can show that the event in Theorem \ref{th:intro} holds with probability exceeding $1-n^{-\gamma}$ for $n$ large enough, at the price of changing the definition of the critical densities $\z_{in}$ and $\z_{out}$ to make them depend on $\gamma$. See Remarks \ref{rem1} and \ref{rem2} for details. 
\end{Remark}

It is worth mentioning an edge case which plays an important role in our proof of Theorem~\ref{th:intro}. Activated Random Walk with sleep rate $\lambda=\infty$ goes by the name \textbf{Internal DLA}. It has the simple description that each random walker continues until it reaches an unoccupied site, which it occupies forever. The case $\lambda=\infty$ of Theorem~\ref{th:intro} follows from the shape theorem of Lawler-Bramson-Griffeath \cite{lawler1992internal}. Vastly more is known about Internal DLA, including fluctuations \cite{asselah2013logarithmic, jerison2012logarithmic} and bounds on mixing \cite{levine2019long,silvestri2020internal}.  We find it remarkable that simply by tuning the sleep rate $\lambda <\infty$, the model becomes so much more difficult that even the shape theorem in one dimension remains unproved!

\subsection{Organization of the paper} Section~\ref{s.defs} gives precise definitions and reviews two basic features of ARW that will be used in the proof of Theorem~\ref{th:intro}, the Abelian and Strong Markov properties.  The inner and outer bounds are proved separately and independently in Sections~\ref{s.inner} and~\ref{s.outer}. The outer bound is proved via a coupling with Internal DLA in a random environment, while the inner bound follows from a block decomposition similar to the one introduced in \cite{asselah2019diffusive,basu2018non}.  Some bounds on variants of Internal DLA will be needed, and we collect their proofs in Appendix~\ref{appendix}.

\section{Precise definitions; Abelian property; Strong Markov property}
\label{s.defs}

Let $\bbN_s := \bbN \cup \{ s \} $ denote the ordered set 
	\[ 0<s<1<2<3\cdots \]
An Activated Random Walks (in short ARW) configuration on $\bbZ$ is a map 
	\[ \eta : \bbZ \longrightarrow \bbN_s \]
where $\eta (x)$ denotes the number of active particles at $x$ if $\eta (x) \geq 1$, a sleeping particle at $x$ if $\eta (x) = s$, and no particle at $x$ if $\eta (x) =0$. 
A particle can only be sleeping if alone at its site, and it gets instantaneously reactivated if an active particle steps on it. 
Set $|s|=1$, so that $|\eta (x)| $ counts the number of particles at $x$, regardless of their state. 
\begin{Definition}
A configuration $\eta $ on $\bbZ$ is called \emph{stable} at $x$ if $\eta (x) \in \{0,s\}$, and it is \emph{unstable} at $x$ otherwise. We say that $\eta $ is \emph{stable} if it is stable at all $x\in \bbZ$. 
\end{Definition}

Next we describe how to stabilize $\eta$ by a sequence of moves called firings. Let $\lambda \in (0,\infty ]$ be a fixed parameter, that we refer to as the \emph{sleep rate}. 
To each site $x \in \bbZ$ we associate an infinite stack $\rho_x = (\rho_x (k))_{k\geq 1}$ of independent instructions with common distribution 
	\begin{equation} \label{instructions}
	 \rho_x (1) = 
	\begin{cases} 
	\; s & \mbox{ with probability } \frac{\l }{1+\l } \\
	x-1 & \mbox{ with probability } \frac{1}{2(1+\l )} \\
	x+1 & \mbox{ with probability } \frac{1}{2(1+\l )} .
	\end{cases}
	\end{equation}
The instructions $s, x-1, x+1$ are interpreted respectively as: ``fall asleep if alone'', ``step left'', ``step right''.  If $x \in \bbZ$ with $\eta(x) \geq 1$ we can \emph{fire} $x$ by applying the first unused instruction in the stack $\rho_x$.  
The effect of firing $x$ is that \emph{one active particle at $x$ steps to a uniform neighbour with probability $1/(1+\l)$, and otherwise it falls asleep if alone} (that is, if no other particles are present at $x$; note that applying an $s$ instruction has no effect if there is more than one particle at $x$).
We denote the resulting configuration by $\tau_x \eta$.  


A firing $\tau_x$ is said to be \emph{legal} for $\eta$ if $\eta (x) \geq 1$. Write $\alpha = (x_1 , x_2 \ldots x_k)$ for a finite sequence of sites, and assume that the corresponding sequence of firings $\tau_{x_1} , \tau_{x_2} \ldots \tau_{x_k}$ is legal, that is $\tau_{x_j}$ is legal for $\tau_{x_{j-1}}  \tau_{x_{j-2}}\cdots \tau_{x_1} \eta $, for all $j\leq k$. Then we write 
	\[ \tau_\alpha \eta := \tau_{x_k}  \tau_{x_{k-1}}  \cdots  \tau_{x_1} \eta  , \]
and say that $\alpha $ \emph{stabilises} $\eta$ if $\tau_\alpha \eta $ is stable.

To each legal sequence of firings $\alpha$ we associate the odometer function $u_\alpha : \bbZ \to \bbN$ defined by 
	\[ u_\alpha (x) = \sum_{j=1}^k \one (x_j = x) , \]
which counts the number of occurrences of $x$ in $\alpha$, or, equivalently, the number of stack instructions used at $x$. 

It will sometimes be useful to fire vertices that contain a sleeping particle. Firing $x$ such that $\eta (x) =s$ consists of forcing the particle at $x$ to wake up and evolve according to the first unused instruction at $x$. Following \cite{rolla2019universality,rolla2015activated}, we say that if $\eta (x) \geq s$ then firing $x$ is \emph{acceptable}. Note that all legal firings are acceptable, while if $\eta (x) =s$ then firing $x$ is acceptable but not legal.

We collect below some useful properties of the ARW dynamics, namely the Abelian Property and the Least Action Principle, for which we refer the reader to \cite{rolla2015activated} (or to \cite[Lemma 4.3]{abelian1} where they are proved in the more general setting of abelian networks).

\begin{Proposition}[\cite{rolla2015activated}, Section 2.2]\label{pr:facts}
For any fixed realization $(\rho_x(k))_{x\in \bbZ , k\geq 1}$ of instruction stacks, the following hold. 
\begin{enumerate}
\item[1.] (Local Abelian Property) If $\alpha , \beta $ are acceptable sequences of firings for a particle configuration $\eta$ such that $u_\alpha = u_\beta$, then $\tau_\alpha \eta = \tau_\beta \eta$.  The final configuration $\tau_\alpha \eta$ depends on $\alpha$ only through $u_\alpha$. 
\item[2.] (Least action principle) If $\alpha$ is a finite acceptable sequence of firings that stabilises  $\eta$, and $\beta$ is any finite legal sequence for $\eta$, then $u_\beta \leq u_\alpha$.
\item[3.] (Abelian Property) If $\alpha , \beta$ are two legal stabilising sequences for $\eta$, then $u_\alpha = u_\beta $, and in particular $\tau_\alpha \eta  = \tau_\beta \eta  $.
\end{enumerate}
\end{Proposition}

Note that, in particular, the above result tells us that the stabilization of a given particle configuration $\eta$ with stacks $\rho$ does not depend on the order of the firings. \\

We will want to perform a random number of firings and argue that, under some conditions, the stacks of unused instructions have the same distribution as the original ones.  This will require showing that ARW stacks satisfy a form of the Strong Markov Property. To state it, let $(\rho_z (k))_{z\in \bbZ, k\geq 1}$ be a collection of independent random variables such that for each $z$ the sequence $(\rho_z(k))_{k \geq 1}$ is i.i.d.\ with distribution \eqref{instructions}. Write $\mathbf{FS}$ for the set of finitely supported functions $f: \bbZ \to \bbN$.  Write $\rho^f$ for the shifted collection
	\[ \rho^f := (\rho_z (k+f(z)))_{z \in \bbZ, k \geq 1}. \]
Let $\cF_0$ denote a $\sigma$-field independent of the $\rho_z (k)$'s, and 
write $\cF_f$ for the $\sigma$-field generated by $\cF_0$ and  $\rho_z(k)$ for $z \in \bbZ$ and $1 \leq k \leq f(z)$.  

\begin{Proposition}[Strong Markov Property For i.i.d.\ Stacks] \label{le:SMP}
Let $F: \bbZ \to \bbN$ be a random function such that $P(F \in \mathbf{FS})=1$ and $\{F=f\} \in \mathcal{F}_f$ for all $f \in \mathbf{FS}$.
On the event $\{F \in \mathbf{FS}\}$, the shifted stacks $\rho^F$ are independent of $(\rho_z (k))_{z \in \bbZ, k \leq F(z)}$, and $\rho^F$ has the same distribution as $\rho$. 
\end{Proposition}

\begin{proof}
Write $\cF_F$ for the $\sigma$-field of events $A$ satisfying $A \cap \{F = f\} \in \cF_f $ for all $f \in \mathbf{FS}$.  We have to show for any finite sequence of distinct pairs $(z_1,k_1), \ldots, (z_n,k_n)$ in $\bbZ \times \bbN$, any sequence of instructions $i_1, \ldots, i_n$, and any event $A \in \cF_F$, it holds
	\[ \bbP (A, \rho^F_{z_1 } (k_1 ) = i_1 , \ldots , \rho^F_{z_n } (k_n ) = i_n ) 
	= \bbP (A) \prod_{j=1}^n \bbP (  \rho_{z_j } (1 ) = i_j ) . \]
Since $\mathbf{FS}$ is countable, 
	\[
	 \bbP (A, \, \rho^F_{z_1 } (k_1 ) = i_1 , 
	  \ldots , \rho^F_{z_n } (k_n ) = i_n ) 
	  = \sum_{f \in \mathbf{FS}} \bbE \left[ \one (A, F = f, \rho^f_{z_1 } (k_1 ) = i_1 , \ldots , \rho^f_{z_n } (k_n ) = i_n )  \right]  \]
where the sum is over all finitely supported $f: \bbZ \to \bbN$. Writing the expectation on the right as the expectation of the conditional expectation $E[ \cdot | \mathcal{F}_f]$, we can pull out the indicator $\one (A , F= f )$, obtaining
	\[ \begin{split}
	 \sum_{f} \bbE & \big[  \one (A , F = f ) 
	 \bbE [ \one (\rho_{z_1} (k_1 + f(z_1) ) = i_1 , \ldots , \rho_{z_n} (k_n + f(z_n)) = i_n ) | \cF_f ] \big]
	 \\ & =  
	 \sum_{f} \bbP (A , F = f ) \prod_{j=1}^n \bbP (  \rho_{z_j } (1 ) = i_j )
	\\ & = \bbP (A) \prod_{j=1}^n \bbP (  \rho_{z_j } (1 ) = i_j ) . \qed
	\end{split} \]
\renewcommand{\qedsymbol}{}
\end{proof}

We will use the Strong Markov Property as follows. Fix a particle configuration $\eta_0 :\bbZ \to  \bbN_s$, and consider a sequence of configurations $(\eta_k)_{k \geq 1}$ in which each $\eta_k$ is obtained by firing $\eta_{k-1}$ at a single site $x_k$, where $x_k$ may depend on the initial condition $\eta_0$, the portion of the stack instructions explored in the first $k-1$ steps, and perhaps some independent randomness, encoded in $\cF_0$. Continue until some stopping time $T$, and let
	\[ F(x) = \# \{1 \leq k \leq T \,:\, x_k=x \} \]
be the number of times $x$ fires during this procedure.  By Proposition~\ref{le:SMP}, if $T$ is a.s. finite then the unexplored stack instructions $\rho^F$ are independent of the explored ones, and $\rho^F$ has the same distribution as $\rho$.

\medskip

\subsection{Notation}
In light of the fact that the stabilization of $\eta$ with sleep rate $\l$ and the corresponding odometer do not depend on the choice of firing sequence, we denote them  by $\cS^\l  (n \d_0 ) $ and $u_n^\l $ respectively. Then for any legal stabilising sequence $\alpha$ for $\eta $ we have $\cS^\l  (n\d_0 ) = \tau_\alpha (n\d_0 )$, and $u_n^\l  = u_\alpha$, where the equalities hold pointwise. Whenever the sleep rate $\lambda$ is fixed we omit it from the notation for brevity. 

For $r>0$ let $B_r (x) = [x-r , x+r ]$ denote the ball of radius $r$ in $\bbZ$ centred at $x$, and write $B_r $ in place of $B_r (0)$ for brevity. 

For a subset $I \subseteq \bbZ$ let $\one_I$ denote the ARW configuration consisting of exactly one active particle at each site of $I$, and no particle outside~$I$. 
If  $I = [a,b] \subseteq \bbZ$ is a finite interval, let 
	\[ \partial I := \{ a-1, b+1 \} \]
denote its (outer) boundary, and 
	\[ \bar I := I \cup \partial I  \]
denote its closure. The cardinality of a finite set $I \subseteq \bbZ$ will be denoted by $ \# I$. 
If $a,b \in  \bbR$  and $I \subseteq \bbZ$, write $I=[a,b]$ or $I=(a,b)$ to mean $I=[a,b] \cap \bbZ$ and $I=(a,b) \cap \bbZ$ respectively. 

Let $\eta$ denote a finite ARW configuration on $\bbZ$, and recall that $|\eta |$ counts the number of particles in $\eta$, regardless of their state. For $I \subseteq \bbZ$ we write $|\eta|_I$ for the number of particles in $\eta \big|_I$, that is 
	\[ |\eta |_I := \sum_{x \in I } |\eta (x) | . \]

Recall that $\cS(\eta )$ denotes the stabilization of $\eta$. For an interval $I \subseteq \bbZ$, with $\cS_I(\eta )$ we denote the stabilization of $\eta $ with killing upon exiting $I$. That is, particles that step outside $I$ die immediately, and are removed from the system. Note that 
	\[ |\cS(\eta )| = |\eta | , \quad \mbox{ while } \quad |\cS_I (\eta )| \leq |\eta | . \]
We will sometimes need to allow particles to start from $\partial I$, in which case we let $\partial I$ act as a sink. If a particle starting on $\partial I$ steps inside $I$ on its first move, then it is killed upon returning to $\partial I$. If it steps outside $\bar I$ on its first move, it is killed immediately. 
All particles starting on $\partial I$, except for the last one, die immediately upon trying to fall asleep on their first move. The last one is allowed to fall asleep on its first move, and it never gets reactivated. This operation is still denoted $\cS_I(\eta )$. 

It will often be useful to partially stabilise a given configuration by only performing firings which are legal in the Internal DLA sense. To be precise, we say that a particle configuration  $\eta$ is IDLA-unstable at $x \in \bbZ$ if $\eta (x) \geq 2$. In that case we fire $x$  by moving one particle according to the first unused instruction at the site. Note that sleep instructions will be ignored, since there is at least another particle at the site. We keep firing IDLA-unstable sites until reaching a configuration with at most one particle at any given site. This IDLA-stable configuration will be denoted by $\cS ^{IDLA}(\eta)$. Note that if all particles were active in $\eta$ then they are all still active in $\cS ^{IDLA}(\eta)$, since we never fire sites with one particle. In a similar manner one can define, for an interval $I \subset \bbZ$, the IDLA stabilization $\cS^{IDLA}_I(\eta )$ of $\eta$ with killing upon exiting $I$. 

Finally, whenever we write that something holds for $n$ large enough we mean that to hold for all $n\geq n_0$, for some $n_0$ which may depend on all other constants introduced before, such as $\z , \l , \g , C_\l , \varepsilon , \d , k_0$ etc., including the implicit $n_0$ of previous instances of the term "large enough".

\section{The inner bound}
\label{s.inner}

Let $\varepsilon  \in (0,1)$ be fixed throughout as in the statement of Theorem \ref{th:intro}.  Assume that ${\z_{in}  <1}$, since if $\z_{in}=1$ then the inner bound of Theorem \ref{th:intro} follows by standard IDLA arguments (cf.\ Lemma~\ref{le:IDLA}). 

\textbf{Overview of the proof.} To prove the inner bound for $A_n = \text{supp}(u_n)$ we will start by finding an interval $I_1$ such that, with high probability, the odometer $u_n$ is positive on $I_1$ and the stabilization $\cS^\lambda (n\d_0)$ has density not much greater than $\zeta_{in} $ on $I_1$. If the length of $I_1$ is less than $n/(\zeta_{in} +\e ) $, then a large number of particles must escape from one or both endpoints of $I_1$. We use these leaked particles to find a larger interval $I_2$ such that, again with high probability, $u_n$ is positive on $I_2$ and $\cS^\lambda (n\d_0)$ has density not much greater than $\zeta_{in}$ on $I_2$. This procedure is repeated a finite number of times (depending on $\varepsilon$) to obtain an interval $I_k$ of length exceeding $  n/(\zeta_{in} +\e )$. \\

To start with, set  
	\begin{equation}\label{I1}
	 I_1 = \left[ -\frac{n}{2} (1-\e ) , \frac n2 (1-\e ) \right] .
	 \end{equation}
The first lemma below tells us that the IDLA stabilization of $\eta = n \d_0$ with killing upon exiting $I_1$ fills $I_1$ with high probability. 
\begin{Lemma}\label{le:IDLA}
For arbitrary $\e \in (0,1)$ define $I_1$ as in \eqref{I1}. Then for any $\d \in (0,1/2)$ and $n$ large enough it holds 
	\[ \bbP \left(\cS^{IDLA}_{I_1} ( n \d_0 )= \one_{I_1} \right) \geq 1-e^{- n^\delta }. \]
\end{Lemma}
This follows from simple martingale arguments; we postpone the proof to Appendix \ref{appendix}. \\

The next result describes the restriction of the ARW stabilization of $n \d_0$ to an interval~$I$ on the event that IDLA fills $I$, using the crucial observation that this restriction depends on the stacks outside $I$ only via the odometer values on $\partial I$. 
\begin{Lemma}[Coupling Between ARW On $\bbZ$ And ARW Killed On Exiting An Interval] \label{le:crucial} 
Fix $n \in \bbN$ and $a,b \in \bbZ$ with $a<0<b$. Let $\cS (n \d_0 ) $ and $u_n $ denote the stabilization of $n \d_0$ and the odometer function respectively, using i.i.d.\ stacks $ \rho = (\rho_x (k))_{x \in \bbZ , k \geq 1}$. 
For an interval $I = [a+1,b-1] \cap \mathbb{Z}$, denote by $F_I$ the event that IDLA, starting from $n \d_0$ and killed upon exiting $I$, using the same stacks $\rho$ fills $I$.
For nonnegative integers $l$ and $r$ let $ l \d_a + \one_I + r \d_b $ denote the configuration consisting of one active particle at each site of $I$, together with $l$ active particles at $a$ and $r$ active particles at $b$ (see Figure \ref{fig:crucial} below). Denote by $\tilde \cS_I ( l \d_a + \one_I + r \d_b ) $ its stabilization with killing on $\partial I$, using i.i.d.\ stacks $\tilde \rho = (\tilde \rho_x (k))_{x \in \bbZ , k \geq 1}$. 
Then there exists a coupling of $\rho$ and $\tilde \rho$ such that for all $l$ and $r$, on the event $F_I \cap \{ u_n(a) = l , u_n (b) = r\} $ it holds 
	\[ \cS (n \d_0 ) \big|_I = \tilde \cS_I (  l \d_a + \one_I + r \d_b ) \big|_I . \]
\end{Lemma}
\begin{figure}[!h] 
  \centering
    \includegraphics[width=\textwidth]{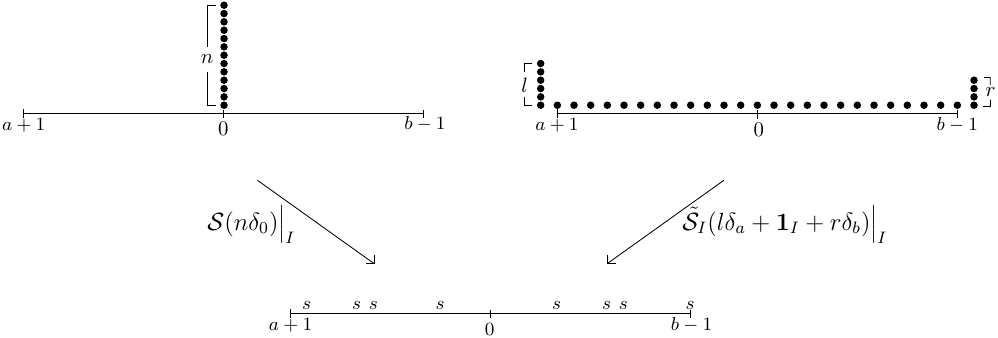}
    \caption{An illustration of the equality $ \cS (n\d_0) \big|_I = \tilde \cS_I ( l\d_{a} + \one_I + r\d_{b} ) \big|_I  $.}\label{fig:crucial}
\end{figure}
\begin{proof}
The result follows from the Abelian Property (cf.\ Proposition \ref{pr:facts}) together with the Strong Markov Property (cf.\ Proposition \ref{le:SMP}). 
We build the stacks $\tilde \rho$ from $\rho$ as follows. Starting with $n$ active particles at the origin and i.i.d.\ stacks $\rho $, move the particles according to IDLA dynamics until they either reach an empty site or reach $\partial I = \{a,b\}$, where they stop. Note that if a moving particle encounters a sleep instruction in this phase it discards it and looks at the next instruction, since if it is moving it cannot be alone at a site. As soon as a particle finds an empty site in $I$ it stops there, while particles that reach $\partial I = \{ a,b\}$ accumulate at $a$ and $b$. Assume that this procedure fills $I$, i.e.\ the event $F_I$ holds, and denote by $u_I$ the associated odometer function, so that $u_I(x)$ counts the number of used instructions at $x$ if $x \in I$, and $u_I(x) = 0$ otherwise. Then set 
	\[ \tilde \rho_x (k) = \begin{cases} 
	\rho_x (k+ u_I(x))  & \mbox{if } x \in I , \\
	\rho_x(k) & \mbox{otherwise.}
	\end{cases} \]
By the Strong Markov Property of ARWs, the shifted stacks $\tilde \rho$ have the same distribution as $\rho$.
It remains to show that, under this coupling, stabilising the two systems will result in the same configuration inside $I$. Recall that in the $\rho$ system the particles move according to standard ARW dynamics on $\bbZ$, while in the $\tilde \rho$ system they are killed on $\partial I$. We therefore refer to these systems as the standard system and killed system respectively. To argue that the coupled stacks give the same stable configuration in $I$ observe that we can evolve the standard and killed dynamics together by firing both systems at the same site $x$ if $x \in \bar I $, and only the standard system if $x \notin \bar I$. In terms of particle trajectories, this is equivalent to saying that each time a particle crosses $\partial I$ in the standard system, a particle gets killed and a new one starts from the same site in the killed system. 
This results in indistinguishable dynamics, and therefore in the same stable configuration, in $I$. 
\end{proof}
In the proof of the inner bound we will combine the above coupling with the following observation.
\begin{Lemma} \label{le:equality}
Fix an arbitrary interval $I = [a+1 , b-1]$ and integers $l,r \geq 0$. Then 
	\[ \cS_{I} ( l\d_{a} + \one_{I} + r\d_{b} ) \big|_{ I} \stackrel{(d)}{=}  \cS_{I} ( \one_{I} )\big|_{I} . \]
\end{Lemma}
\begin{proof}
Note that to build $\cS_{I} (l\d_{a} + \one_{I} + r\d_{b} )\big|_{ I}$ we can first move the particles starting on $\partial I = \{ a,b\}$ (if any). If they fall inside $I$ they will walk until returning to $\partial I$, at which point they die. If they fall outside $I$ they can never enter $I$ again, as they are killed on $\partial I$. Once we have let all the particles starting on $\partial I$ either die or settle (i.e.\  fall asleep at an empty site) we are left with having to stabilize $\one_{I}$ inside $I$. Hence the above equality in distribution follows by the Strong Markov Property of ARWs (cf.\ Proposition \ref{le:SMP}). 
\end{proof}

We take $I = I_1 $ as in \eqref{I1}, and build $\cS (n\d_0)\big|_{I_1}$ as follows. First move the particles according to IDLA in $I_1$ killed at $\partial I_1$, which will result in the configuration $\one_{I_1}$ with 
high probability by Lemma \ref{le:IDLA}. Then, 
on the event $\{ u_n (a) = l , u_n(b) = r\}$, we stabilize $l \d_a + \one_{I_1} + r \d_b $ with killing outside $I_1$. To implement this program we need an a priori bound on the ARW odometer function.

\begin{Lemma}[A Priori Upper Bound On The Total Odometer] \label{le:apriori}
	\[ \bbP \left( \sum_{x\in \bbZ} u_n(x) > n^4 \right) \leq n^{-2} \]
for $n$ large enough. 
\end{Lemma}
\begin{proof}
By the least action principle it suffices to exhibit a stabilization procedure with total odometer bounded by $n^4$ with high probability. To this end, note that forcing particles to wake up can only increase the total odometer. We therefore proceed as follows. Let $G$ denote a geometric random variable with parameter $\l /(1+\l )$, and choose $C_\l $ so that 
	\[ \bbP (G > 2 C_\l \log n ) \leq n^{-7/2} \]
for $n$ large enough. Set $\D = 2 \lceil C_\l \log n \rceil $. We release $n$ particles from the origin one at the time. Each particle performs a simple random walk on $\bbZ$, until reaching an unoccupied site of the $\D$-grid $\{ k \D : k\in \bbZ\}$ where it stops, thereby changing the site into occupied. Note that this is equivalent to running IDLA on $\bbZ$ with particles only allowed to stop on the $\D$-grid. At the end of this procedure all particles are at distance $\D$ one another. We release them, evolving the particle system according to ARW dynamics. By the choice of $\D$ all the particles will fall asleep before they are able to interact, which entails stabilization. 
\begin{figure}[!h] 
  \centering
    \includegraphics[width=\textwidth]{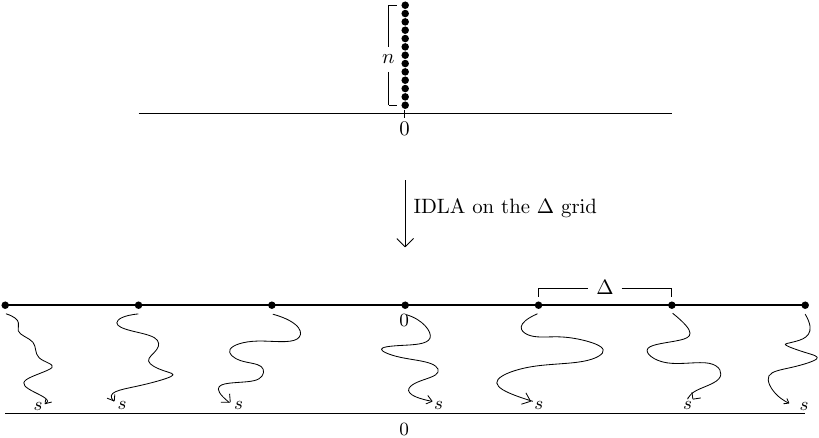}
    \caption{An illustration of the proof. The particles are spread enough so that their trajectories do not intersect before they fall asleep.}
\end{figure}

This procedure takes total odometer 
	\[ \sum_{k=1}^n (T_k + G_k ) , \]
where $T_k$ and $G_k$ denote the time it takes for the $k^{th}$ walk to occupy a site of the $\D$-grid, and to fall asleep when released, respectively. Note that 
	\[ \bbP \left( \max_{1\leq k \leq n} G_k \geq \D  \right) \leq n^{-5/2} \]
for $n$ large enough, by the choice of $\D$. Moreover, if $\tilde T_k$ denotes the time it takes for the $k^{th}$ random walk to reach distance $k\D$ from the origin, we have that 
	\[  \sum_{k=1}^n T_k \leq  \sum_{k=1}^n \tilde T_k , \]
with the advantage that the $\tilde T_k$'s are independent random variables. We use that by \cite{lawler1992internal}, Proposition 2.4.5, 
	\[ \bbP \left( \tilde{T}_k \geq (k\D \log n )^2 \right) \leq n^{-5/2} \]
for all $k\leq n$, and $n $ large enough. Thus 
	\[ \begin{split} 
	\bbP \left( \sum_{k=1}^n (T_k + G_k ) > n^4 \right) & \leq 
	\bbP \left( \sum_{k=1}^n \tilde{T}_k + n \D > n^4 \right) + n^{-5/2} 
	\\ & \leq \one \left( \sum_{k=1}^n (k\D \log n )^2 + n\D > n^4 \right) + 2n^{-5/2} 
	\leq  n^{-2} 
	\end{split}\]
for $n$ large enough, which concludes the proof. 
\end{proof}
Let $B$ (as in \emph{bound}) denote the event of the above lemma, and $F$ (as in \emph{fill}) denote the event of Lemma \ref{le:IDLA}.  We are now in position to show that $\cS (n\d_0)$ has low density on~$I_1$ with high probability. 
\subsection*{Step 1}
Recall that for a stable ARW configuration $\eta $ on $\bbZ$ 
	\[ |\eta | := | \{ x : \eta (x) >0 \}| \]
counts the number of sleepers in $\eta$. 
Suppose that $\d \in (0,1)$ is fixed small enough so that $\z_{in} + \d <1$. We will choose $\d$ later depending on $\z_{in}$ and $\e$.  
Define the good event 
	\begin{equation}\label{S1}
	 E_1 := B \cap F \cap \{  |\cS (n\d_0 ) |_{I_1}  \leq (\z_{in} +\d ) \# I_1  \} . 
	 \end{equation}
\begin{Lemma}\label{le:step1}
For $n$ large enough 
	\[ \bbP (E_1^c ) \leq n^{-3/2} . \]
\end{Lemma}
\begin{proof}
We have 
	\[ \bbP(E_1^c) \leq \bbP(B^c) + \bbP(F^c) + \bbP \big( B\cap F \cap 
	\{  |\cS (n\d_0 )|_{ I_1}   > (\z_{in} +\d )  \# I_1  \} \big) . \]
By Lemmas \ref{le:IDLA} and \ref{le:apriori}
	\[ \bbP(B^c) + \bbP(F^c) \leq 2n^{-2} , \]
so it remains to bound the third term above. 
Write $I_1 = [a_1 +1 , b_1 -1]$ so that $\partial I_1 = \{ a_1,b_1\}$. We union bound over all possible values of the odometer function  on $\partial I_1$ to find
	\[ \begin{split} 
	\bbP \big( B\cap F & \cap \{  |\cS (n\d_0 )  |_{ I_1}   > (\z_{in} +\d )  \# I_1  \} \big) = 
	\\ & = \sum_{l,r=0}^{n^4} \bbP \big( B \cap F \cap \{u_n( a_1) =l \}\cap \{ u_n(b_1)=r \} \cap \{  |\cS (n\d_0 )|_{ I_1}   > (\z_{in} +\d )  \# I_1 \} \big) 
	 \\ & \leq  \sum_{l,r=0}^{n^4} \bbP \big( | \cS_{I_1} (l\d_{a_1} + \one_{I_1} + r\d_{b_1} )|_{ I_1} > (\z_{in} +\d )  \# I_1  \big) 
	 \\ & = n^8  \bbP \big( |\cS_{I_1} (\one_{I_1} ) | > (\z_{in} +\d )  \# I_1  \big) .
	\end{split}\]
Here the inequality follows from Lemma \ref{le:crucial}, and the last inequality from Lemma \ref{le:equality}. 

Now by Definition~\ref{def:zin} we can take $ \# I_1$ (or equivalently $n$) large enough that $\z_{in} +\d \geq \z_{in,I_1} + \d /2$, which gives 
	\[  \bbP \big( |\cS_{I_1} (\one_{I_1} ) | > (\z_{in} +\d )  \# I_1  \big)
	\leq 
	 \bbP \big( |\cS_{I_1} (\one_{I_1} ) | > (\z_{in, I_1} +\d /2 )  \# I_1  \big)
	 \leq ( \# I_1)^{-20} ,\]
and thus 
	\[  \bbP (E_1^c ) \leq 2 n^{-2} + n^8 (  \# I_1)^{-20} \leq n^{-3/2}\]
for $n$ large enough.  
\end{proof}

\subsection*{General step} 
We will define events $E_1 \supseteq E_2 \supseteq E_3 \supseteq \cdots $ by induction. 
For $j \geq 2 $ assume  $E_{j-1}$ has been defined so that on $E_{j-1}$ there is an interval $I_{j-1} \supseteq I_1$ such that 
	\[ |\cS (n\d_0)|_{ I_{j-1}}  \leq (\z_{in} + \d )\# I_{j-1}  . \]
If 
	\[ \# I_{j-1} \geq  \frac{n}{\z_{in}+\d } (1-3\e )\] 
set $K = j-1$ and stop (here the random variable $K$ will denote the number of steps in the stabilization procedure). Otherwise $K \geq  j$ and 
	\begin{equation}\label{notdone}
	 n -  |\cS (n\d_0)|_{ I_{j-1}} \geq 2n\e . 
	 \end{equation}
On the event $E_{j-1} \cap \{ K \geq j\}$ write $I_{j-1} = (a_{j-1} , b_{j-1} )$ for integers $a_{j-1 } , b_{j-1} $ in the interval $ \in [-n/(\z_{in} +\d ) , n/(\z_{in} +\d )]$. 
Then we define $I_j \supseteq I_{j-1}$ as follows. 
For integers $r_- , l_+ \leq n^4$ on the event 
	\[E_{j-1} \cap \{ K \geq j\} \cap  \{  u_n(a_{j-1}) = r_- , u_n(b_{j-1}) = l_- \} \]
let $I_j ^\pm $ be defined by 
	\[ \cS^{IDLA}_{(-\infty , a_{j-1} )} (r_- \d_{a_{j-1}} ) = \one_{I_j^-} , \qquad 
	  \cS^{IDLA}_{(b_{j-1} , +\infty )} (l_- \d_{b_{j-1}} ) = \one_{I_j^+} , \]
	where the stabilizations use the original stacks of instructions. 
If $|I_j^+ | < n\e $ (respectively $|I_j^- | < n\e $)  then redefine $I_k^+ = \emptyset $ (respectively $I_k^- =\emptyset$) for all $k\geq j$. Note that, by \eqref{notdone}, there are at least $2n\e $ particles outside $I_{j-1}$ in $\cS (n\d_0)$, so at least one of $I_j^+ $ and $I_j^-$ must be non-empty. Then
define 
	\begin{equation} \label{Ij}
	I_j = I_j^- \cup \bar I_{j-1} \cup I_j^+ , 
	\end{equation}
and 
	\begin{equation} \label{Sj}
	 E_j = E_{j-1} \cap  \{ |\cS (n\d_0 ) |_{  I_j \setminus \bar{I}_{j-1}}  
	  \leq (\z_{in} +\d )  \# ( I_j \setminus \bar{I}_{j-1}) \} . 
	  \end{equation}
Thus on $ E_j \cap \{ j \leq K\}$ the stable configuration $\cS (n\d_0 )$ has  density 
at most $\z_{in} + \d + \mathcal{O}(j/n)$ on the interval~$I_j$ (here the term $\mathcal{O}(j/n)$ accounts for the boundary sites). 
\begin{figure}[!h] 
  \centering
    \includegraphics[width=\textwidth]{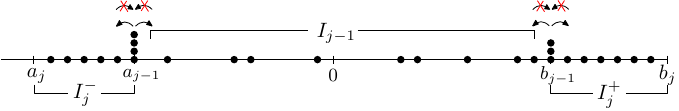}
    \caption{An illustration of the construction of $I_j$ from $I_{j-1}$.}
\end{figure}

\begin{Lemma}\label{le:genstep}
For all $j\geq 2$
	\[ \bbP( E_{j-1} \cap E_j^c\cap \{ j\leq K\} ) \leq n^{-3/2}\]
for $n$ large enough. 
\end{Lemma}
\begin{proof}
We use the union bound to fix the intervals $I_{j}^\pm  $ together with the odometer values on their boundaries, at which point the result follows by similar arguments to those of  Lemma \ref{le:step1}. 
Fix sites $ a_{j-1} \leq 0 \leq b_{j-1} $ with $\max \{ -a_{j-1} , b_{j-1} \} \leq n/(\z_{in} +\d )$, and odometer values $0 \leq r_- , l_+ \leq n^4 $. Then we have
	\[ \begin{split} 
	& \bbP  (E_{j-1}  \cap E_j^c \cap \{ j\leq K \} \cap \{ I_{j-1} = (a_{j-1} , b_{j-1 }), 
	u_n(a_{j-1}) = r_-  , u_n(b_{j-1}) = l_+ \} ) 
	\\ & \leq 
	\sum_{a_{j} } 
	\bbP \big(u_n (a_{j-1} )= r_- ,  \cS^{IDLA}_{(-\infty , a_{j-1} )} (r_- \d_{a_{j-1}} ) = \one _{(a_j , a_{j-1})} , 
	|\cS (n\d_0 )|_{(a_j , a_{j-1})} > (\z_{in} + \d ) |a_j - a_{j-1}| \big)
	 \\ & \;  + 
\sum_{b_j}
	\bbP \big(u_n (b_{j-1})= l_-  ,  \cS^{IDLA}_{( b_{j-1}, +\infty )} (l_+ \d_{b_{j-1}} ) = \one _{(b_{j-1} , b_{j})} , 
	|\cS (n\d_0 )|_{(b_{j-1} , b_{j})} > (\z_{in} + \d ) |b_j - b_{j-1}| \big) .
	\end{split}\]
We bound each term in the first summation as follows. Denote by $\rho = (\rho_k)_{k\in \bbZ}$ the underlying stacks of instructions, and write $u^\infty_{r_-}$ for the odometer function associated to the stabilization $\cS^{IDLA} _{(-\infty , a_{j-1})} (r_- \d_{a_{j-1}}) $. Then $u^\infty_{r_-}$ is supported in $[a_{j-1} - n^4 , a_{j-1}]$, since the $r_- $ particles can fill an interval of length at most $r_- \leq n^4$ to the left of $a_{j-1}$. 
We define the shifted stacks of instructions $\tilde \rho = (\tilde{\rho }_k )_{k\in \bbZ}$ by setting 
	\[ \tilde \rho_x (k) = \begin{cases} 
	\rho_x (k+ u^\infty_{r_-} (x))  & \mbox{if } x \leq a_{j-1}  , \\
	\rho_x(k) & \mbox{otherwise.}
	\end{cases} \]
Then, by Lemma \ref{le:crucial}, on the event $\{ u_n (a_{j-1} )= r_- ,  \cS^{IDLA}_{(-\infty , a_{j-1} )} (r_- \d_{a_{j-1}} ) = \one _{(a_j , a_{j-1})} \}$ we have that 
	\[ \cS (n\d_0 ) |_{(-\infty , a_{j-1} ) } 
	= \tilde \cS_{(-\infty , a_{j-1})} ( \one_{(a_j , a_{j-1}) } ) , \]
where $\tilde \cS_I (\eta )$ denotes the stabilization of $\eta$ according to stacks $\tilde \rho $ with killing upon exiting $I$. This gives 
	\[ \begin{split} 
	\bbP & \big(u_n (a_{j-1} )= r_- ,  \cS^{IDLA}_{(-\infty , a_{j-1} )} (r_- \d_{a_{j-1}} ) = \one _{(a_j , a_{j-1})} , 
	|\cS (n\d_0 )|_{(a_j , a_{j-1})} > (\z_{in} + \d ) |a_j - a_{j-1}| \big)
	\\ & \leq \bbP \big( \cS^{IDLA}_{(-\infty , a_{j-1} )} (r_- \d_{a_{j-1}} ) = \one _{(a_j , a_{j-1})} , 
	|\tilde \cS_{(-\infty , a_{j-1})} ( \one_{(a_j , a_{j-1}) } )|_{(a_j , a_{j-1})} > (\z_{in} + \d ) |a_j - a_{j-1}| \big)
	\\ & \leq \bbP \big( |\cS_{(-\infty , a_{j-1})} ( \one_{(a_j , a_{j-1}) } )|_{(a_j , a_{j-1})} > (\z_{in} + \d ) |a_j - a_{j-1}| \big) , 
	\end{split}\]
where the last inequality follows from the Strong Markov Property. 
Similarly one shows that 
	\[ \begin{split} 
	\bbP & \big(u_n (b_{j-1})= l_-  ,  \cS^{IDLA}_{( b_{j-1}, +\infty )} (l_+ \d_{b_{j-1}} ) = \one _{(b_{j-1} , b_{j})} , 
	|\cS (n\d_0 )|_{(b_{j-1} , b_{j})} > (\z_{in} + \d ) |b_j - b_{j-1}| \big) 
	\\ & \leq \bbP \big( |\cS_{(b_{j-1} , +\infty )} ( \one_{(b_{j-1} , b_{j}) } )|_{(b_{j-1} , b_{j})} > (\z_{in} + \d ) |b_j - b_{j-1}| \big) . 
	\end{split}\] 
From this, with a similar argument to that of Lemma \ref{le:step1}, we get
	\[ \begin{split} 
	\bbP & (E_{j-1}  \cap E_j^c \cap \{ j\leq K \} \cap \{ I_{j-1} = (a_{j-1} , b_{j-1 }), 
	u_n(a_{j-1}) = r_-  , u_n(b_{j-1}) = l_+ \} ) 
	\\ & \leq 
	\sum_{a_j \leq n^4 } \bbP \big( B \cap \{ | \cS_{(-\infty , a_{j-1} )} (\one_{(a_j , a_{j-1})} ) |_{(a_j , a_{j-1} )} 
	> (\z_{in} +\d ) |a_j - a_{j-1}|  \} \big) 
	\\ & \;  + 
	\sum_{b_j \leq n^4 } \bbP \big( B \cap \{ | \cS_{(b_{j-1} , +\infty )  } (\one_{(b_{j-1} , b_j )} ) |_{(b_{j-1} , b_j )} 
	> (\z_{in} +\d ) |b_j - b_{j-1}| \} \big) 
	\\ & = 
	\sum_{a_j , l_-  \leq n^4 } \bbP \big( B \cap \{ u_n (a_j ) = l_- \} \cap \{  | \cS_{(a_j , a_{j-1} )} (\one_{(a_j , a_{j-1})} + l_- \d_{a_j}  ) |_{(a_j , a_{j-1} )} 
	> (\z_{in} +\d ) |a_j - a_{j-1}| \} \big) 
	\\ & \;  + 
	\sum_{b_j , r_+ \leq n^4 } \bbP \big( B \cap \{ u_n (b_j ) = r_+ \} \cap \{  | \cS_{( b_{j-1}, b_j  )} (\one_{(b_{j-1} , b_j )} + r_+ \d_{b_j} ) |_{(b_{j-1} , b_j )} 
	> (\z_{in} +\d ) |b_j - b_{j-1}| \} \big) 
	\\ & \leq 
	\sum_{a_j , l_- \leq n^4 } \bbP \big( | \cS_{(a_j , a_{j-1} )} (\one_{(a_j , a_{j-1})}  ) |_{(a_j , a_{j-1} )} 
	> (\z_{in} +\d ) |a_j - a_{j-1}| \big) 
	\\ & \; + 
	\sum_{b_j , r_+ \leq n^4 } \bbP \big( | \cS_{( b_{j-1}, b_j  )} (\one_{(b_{j-1} , b_j )} ) |_{(b_{j-1} , b_j )} 
	> (\z_{in} +\d ) |b_j - b_{j-1}| \big) 
	\\ & \leq 2 n^8 (n\e )^{-20} = 2 \e^{-20} n^{-12}
	\end{split}\]
for $n$ large enough. Here again the third inequality follows by evolving all the particles starting on boundary sites until they die, as in Lemma \ref{le:equality}. Moreover, in the last inequality we used that $l_- , r_+ \leq n^4$, so there are at most $n^8$ choices for the pair $(a_j , b_j) $. 
This gives 
	\[ \begin{split} 
	\bbP & ( E_{j-1}  \cap E_j^c\cap \{ j\leq K\} ) 
	\\ & \leq \sum_{a_{j-1} , b_{j-1}} \sum_{r_- , l_+} 
	\bbP  (E_{j-1}  \cap E_j^c \cap \{ j\leq K \} \cap \{ I_{j-1} = (a_{j-1} , b_{j-1 }), 
	u_n(a_{j-1}) = r_-  , u_n(b_{j-1}) = l_+ \} ) 
	\\ & \leq \frac{2\e^{-20} }{(\z_{in} +\d )^2} n^{-2} \leq n^{-3/2}  
	\end{split}\]
for $n$ large enough, since on the event $\{ j \leq K \} $ there are at most $n/(\z_{in} +\d ) $ choices for both $a_{j-1}$ and $b_{j-1}$.
\end{proof}
Now, on the event $E_j \cap \{ j\leq K\}$ we have 
	\[ | \cS (n\d_0 )|_{I_j} \leq ( \z_{in} +\d ) \# I_j  + 2j -2 ,  \]
where the additional term $2j-2$ comes from the boundary sites. 
If 
	\[ \# I_j \geq \frac{n}{\z_{in} +\d } (1-3\e ) \]
set $K=j$ and stop, setting $E_k = E_j $ for all $k\geq j$. Otherwise $K\geq j+1$ and we move on to the next step.

\subsection*{Closing the iterative scheme}
It remains to bound the number of iterations. Write $L_j =  \# I_j$ for all $j\geq 1$. Let $k_0\geq 2$ be and integer to be chosen later, and recall that $K$ was defined in the previous subsection as 
	\begin{equation}\label{Kdef}
	 K:= \inf \left\{ j\geq  2 :  L_j \geq \frac{n}{\z_{in}+\d }(1-3\e ) \right\} . 
	 \end{equation}
Then on the event $ E_{k_0} \cap \{ K > k_0\}$ the following inequalities hold for $n$ large enough:
	\[ \begin{split} 
	L_1 & = 2\left\lfloor \frac n2 (1-\e ) \right\rfloor + 1 \geq n(1-\e ) -1 \geq n(1-2\e ), \\
	L_j &  \geq L_{j-1} + n - |\cS (n\d_0) |_{ I_{j-1}} -\e n 
	\\ & \geq L_{j-1} + n (1-2\e ) - (\z_{in} + \d ) L_{j-1}   , \qquad 2 \leq j \leq k_0 ,
	\end{split} \]
where the $-\e n $ extra term in the first lower bound for $L_j$ accounts for the fact that one of $I_j^+$ and $I_j^-$ might be empty. 
Now, using that for all $j \leq k_0$ one has $L_{j-1} \leq \frac{n}{\z_{in} + \d } (1+\e )$, we find $ n (1-2\e ) - (\z_{in} + \d ) L_{j-1}  \geq \e n $, 
which shows that 
	\[ L_j \geq L_{j-1} + n\e \geq n ( 1 + (j-3) \e ) \] 
for all $2\leq j \leq k_0$. 
Choose $k_0$, depending only on $\z_{in}$, $\d $ and $\e$, such that 
	\[ 1 + (k_0 -3 ) \e \geq \frac{1-3\e }{\z_{in} + \d } .\]
Then
	\[ L_{k_0} \geq  \frac{n(1-2\e )(1-\e )}{\z_{in} + \d } \geq \frac{n}{\z_{in}+\d }(1-3\e ) ,  \]
and hence event $\{ K >  k_0 \} \cap E_{k_0} $ is empty. Then 
	\[ \begin{split} 
	\bbP (K > k_0 ) & = \bbP ( \{ K > k_0 \} \cap ( E_1^c \cup E_2^c \cup \cdots \cup E_{k_0}^c ) )
	\\ & \leq \bbP (E_1^c) + \sum_{j=2}^{k_0} \bbP (E_{j-1}  \cap E_j^c \cap \{ K\geq j\}) 
	\leq (k_0+1) n^{-3/2} . 
	\end{split}\]
To close the proof of the inner bound, take $ \d \in \left( 0 ,  \e -3\e^2 (1-\z_{in} ) \right] $,
so that 
	\[ L_{K} \geq    \frac{n}{\z_{in}+\d }(1-3\e ) \geq \frac{n}{\z_{in} +\e }  .  \]
Then 
	\[ \left\{ |\mbox{supp} (u_n)| \geq \frac{n}{\z_{in} +\e }  \right\} 
	\supseteq  \{ K \leq k_0\}  , \]
and so 
	\[ \bbP \left(  |\mbox{supp} (u_n)| < \frac{n}{\z_{in} +\e }  \right) 
	\leq \bbP ( K>k_0 ) \leq (k_0 +1 ) n^{-3/2} \]
which is summable in $n$. 
\begin{Remark}
A more careful use of the above argument shows that in fact one only needs a logarithmic number of steps in $1/\e $ in order to close the iterative scheme. We leave the details to the reader.
\end{Remark}
\begin{Remark} \label{rem1}
With the same arguments we can get an inner bound that holds with higher probability in $n$. More precisely, for any $\g >0$ we can define 
	\[ \z_{in, I} (\g ) := \inf \left\{  \z >0 : \bbP (|\cS_I(\one_I)| > \z  \# I ) \leq ( \# I)^{-(\g + 20) } \right\} , \]
and 
	\[ \z_{in}(\g ) := \limsup_I \z_{in,I}(\g )  , \]
to have that for any $\e >0$ 
	\[ \bbP \left( \# A_n \geq \frac{n}{\z_{in}(\g )} (1-\varepsilon ) \right) \geq 
	1-n^{-\g } \]
for $n$ large enough. 
\end{Remark}

\section{The outer bound}
\label{s.outer}

To prove the outer bound in Theorem \ref{th:intro} we show that the ARW odometer $u_n$ is stochastically dominated by the sum of two auxiliary odometer functions, corresponding to an IDLA process on a Bernoulli site percolation, and to the ARW stabilization of a subcritical Bernoulli configuration on $\bbZ$. 

Fix $\z \in (0,1)$, and enlarge the probability space to include an independent Bernoulli$(\z )$ site percolation $\eta \in \{ 0,1\}^{\bbZ}$, where each site $x \in \bbZ$ is declared open ($\eta(x) =1$) with probability $\z $, and closed ($\eta (x) =0$) with probability $1-\z $. Write $\bfP_\z$ for the product measure $\bbP \times \mu_\z$, where $\bbP$ is the law of the stacks of ARW instructions \eqref{instructions} with sleep rate $\lambda$, and $\mu_\z =  \prod_{x \in \bbZ} ((1-\z)\delta_0 + \z \delta_1)$ is the law of the site percolation $\eta$. Write $\bfE_\z$ for expectation with respect to $\bfP_\z$.

We interpret $\eta $ as an ARW particle configuration, consisting of exactly one active particle at each open site. Denote by $w : \bbZ \to \bbN \cup \{ \infty \}$ the $ARW^\lambda$ odometer of $\eta $ on $\bbZ$. 

The next result bounds the odometer $w$ under $\bfP_\z$ on large intervals, for $\z < \z_{out}$. 
\begin{Lemma}\label{le:w}
For any $\z \in (0, \z_{out})$ we have  
	\[ \bfP_\z \left( \sup_{x\in B_r} w(x) > \frac{r}{(\log r)^3} \right) \leq r^{-3/2} \]
for $r$ large enough. 
\end{Lemma}
\begin{proof}
Since $\z< \z_{out}$ we have that  $\bfE_\z ( w(0)^3 ) \leq C$ for some finite constant $C>0$. Hence 
	\[ \begin{split} 
	\bfP_\z \left( \sup_{x\in B_r} w(x) > \frac{r}{(\log r)^3} \right)
	 & \leq \sum_{x \in B_r } \bfP_\z \left( w(x) > \frac{r}{(\log r)^3} \right)
	 \leq (  \# B_r ) \frac{(\log r)^9}{r^3}  \bfE_\z (w(0)^3 ) 
	 \\ & \leq C (2r +1)\frac{(\log r)^9}{r^3}
	 \leq r^{-3/2} 
	 \end{split}\]
for $r$ large enough. 
\end{proof}
We are now ready to explain the proof of the outer bound, which consists in stabilizing the configuration $ n\d_0$ in two phases. To start with, sample an i.i.d.\ Bernoulli$(\z )$ site percolation configuration on $\bbZ$, with 
	\[ \z = \z_{out} - \d, \]
$\d$ small to be chosen later. 
In the initial $\z$-IDLA phase we move the $n$ particles one at the time: each particle walks until it finds an unoccupied open site, where it stops. Note that this is equivalent to IDLA dynamics with particles only allowed to stop at open sites. Thus in this phase if a particle tries to fall asleep before reaching an unoccupied open site we force it to wake up, which, by Proposition \ref{pr:facts}, can only increase the odometer. 
\begin{Definition}[$\z$-IDLA odometer function]
We let $v_n : \bbZ \to \bbN $ denote the odometer function of the $\z$-IDLA phase, that is $v_n(x)$ counts the number of instructions used at $x$ until each particle finds its spot. 
\end{Definition}
We remark that the odometer function counts the number of sleep instructions too, although they have no effect on the IDLA dynamics. 
The next lemma tells us that $v_n$ is supported in a ball of radius $\frac n {2\z_{out}} (1+\e )$, with high probability. 
\begin{Lemma}\label{le:vn}
For any $\e >0$
	\[ \bfP_\z \left( \mbox{supp}(v_n) \subseteq B_{\frac{n}{2\z_{out}} (1+\e ) }
	\right) \geq 1 - e^{-n / (\log n)^2}  \]
for $n$ large enough. 
\end{Lemma}
This follows from Theorem \ref{th:IDLABernoulli} in Appendix \ref{appendix}, by taking $\z = \z_{out} - \d$ and choosing $\d = \d (\z_{out} , \e )$ such that 
	\[ \frac{n}{2(\z_{out} -\d )} (1+2\e ) = \frac{n}{2\z_{out} } (1+\e ) .\] 
\begin{figure}[!h] \label{fig:IDLABernoulli}
  \centering
    \includegraphics[width=.9\textwidth]{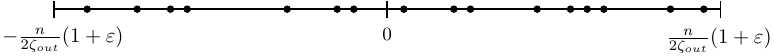}
    \caption{An illustration of the particle configuration at the end of the $\z$-IDLA phase.}
\end{figure}

At the end of the $\z$-IDLA phase we are left, with high probability, with a configuration of $n$ active particles, supported on the open sites of the ball $B_{\frac{n}{2\z_{out}} (1+\e ) } $. It only remains to stabilize such a configuration.  By Abelianness of ARW, assuming that each open site of the Bernoulli configuration $\eta$ on $\bbZ$ contains an active particle cannot decrease the ARW odometer. Hence we obtain
	\begin{equation}\label{pointwise}
	 u_n  \leq v_n + w , 
	 \end{equation}
with the inequality holding pointwise in both the site $x$ and the realization of the Bernoulli configuration $\eta$. 
\begin{Lemma} \label{le:sup_u}
Let 
	\begin{equation}\label{rr}
	 r_- = \frac{n}{2\z_{out}} (1+\e ) , \qquad r = \frac{n}{2\z_{out}} (1+3\e ) . 
	 \end{equation}
Then  
	\[ 
	\bbP \left( \sup_{x\in B_{r} \setminus B_{r_-}} u_n(x) >  \frac {r}{(\log r)^3} \right) 
	\leq n^{-5/4}
	\]
for $n$ large enough. 
\end{Lemma}
\begin{proof}
Note that by \eqref{pointwise}
	\[  \big\{ \mbox{supp}(v_n) \subseteq B_{r_- } \big\} 
	\subseteq  \big\{ u_n(x) \leq w(x) \, \mbox{ for all } |x| > r_-
	\big\} .   \]
We combine this with Lemma \ref{le:w}, which bounds $w$, to get 
	\[ \begin{split} 
	\bbP \left( \sup_{x\in B_{r} \setminus B_{r_-}} u_n(x) >  \frac {r}{(\log r)^3} \right) 
	&= \bfP_\z \left( \sup_{x\in B_{r} \setminus B_{r_-}} u_n(x) >  \frac {r}{(\log r)^3} \right) 
	\\ &  \leq \bfP_\z \left( \mbox{supp}(v_n) \nsubseteq B_{r_-}
	\right) + \bfP_\z \left( \sup_{x\in B_{r}} w(x) > \frac{r}{(\log r)^3} \right) 
	\\ & \leq e^{-n / (\log n)^2}  + r^{-3/2} \leq n^{-5/4}
	\end{split}\]
for $n$ large enough. In the first equality we have used that the odometer function $u_n$ is independent of the Bernoulli percolation configuration. 
\end{proof}

It only remains to show that on the event 
	\[  \sup_{x\in B_{r} \setminus B_{r_-}} u_n(x) \leq   \frac {r}{(\log r)^3}
	\]
it is unlikely that $\max \{ u_n(-r) , u_n(r) \}>0$. 
Let $r_-, r$ be defined as in \eqref{rr} above, so that $r-r_- = n\e / \z_{out}$. We restrict to the event 
	\[  \sup_{x\in B_{r} \setminus B_{r_-}} u_n(x) \leq   \frac {n}{(\log n)^2} , \]
which holds with $\bbP$-probability at least $1-n^{-5/4}$ for $n$ large enough. On this event 
	\[ \max \{ u_n(-r_-) , \, u_n(r_- ) \} \leq  \frac {n}{(\log n)^2} . \]
To conclude the proof of the outer bound we will use the following fact, which is a one--sided version of Lemma \ref{le:crucial}. 
\begin{Lemma}\label{le:one_side}
Fix an integer $a>0$. Let $\cS (n\d_0)$ and $u_n$ denote the ARW stabilization of $n\d_0$ and the odometer function respectively, using i.i.d.\ stacks $\rho = ( \rho_x(k) )_{x\in \bbZ , k\geq 1}$. For $N\geq 0$ integer denote by $\tilde{\cS}_{(a, +\infty )} (N\d_a )$ the stabilization of the particle configuration $N \d_a $ with killing on $(-\infty , a ] $, with i.i.d.\ stacks $\tilde{\rho} = ( \tilde{\rho}_x(k) )_{x\in \bbZ , k\geq 1}$. Then there exists a coupling of $\rho $ and $\tilde \rho$ such that for all $N\geq 0$ on the event $\{ u_n(a) = N\}$ it holds 
	\[ \cS (n\d_0 ) \big|_{(a , +\infty )} = \tilde \cS_{(a,+\infty )} (N\d_a ) \big|_{(a,+\infty )} . \] 
\end{Lemma}
This can be proved using the same argument of Lemma \ref{le:crucial}, we leave the details to the reader. 
The outer bound then follows from the next result. 
\begin{Lemma}
With $r$ defined as in \eqref{rr}, it holds  
	\[
	 \bbP \big( \max\{  u_n (-r) , u_n (r) \} >0 \big) \leq 2n^{-5/4 }
	\]
for $n$ large enough. 
\end{Lemma}
\begin{proof}
By Lemma \ref{le:sup_u}
	\[ \begin{split} 
	\bbP \big( \max\{  u_n (-r) , u_n (r) \} >0 \big) 
	& \leq \bbP \left( u_n (r) >0 , \; u_n(r_- ) \leq  \frac {n}{(\log n)^2} \right)
	 \\ & \quad + \bbP \left( u_n (-r) >0 , \; u_n(-r_- ) \leq  \frac {n}{(\log n)^2} \right) + n^{-5/4} .
	 \end{split}
	 \]
We prove that 	
	\[ \bbP \left( u_n (r) >0 , \; u_n(r_- ) \leq  \frac {n}{(\log n)^2} \right) \leq n^{-2} \]
for $n$ large enough, since then the same holds for the specular event. To see the above inequality, note that by Lemma \ref{le:one_side} we can build the stable configuration $\cS (n\d_0)$ on $(r_- , + \infty )$ as follows. Let $N \leq n/(\log n)^2$ be arbitrary but fixed. On the event $ u_n (r_- ) = N $, release $N$ particles from $r_- $ according to the following rules: 
\begin{itemize}
\item[(i)] Particles that move to the left of $r_-$ on the first step die. 
\item[(ii)] Particles that move to the right of $r_-$ on the first step evolve according to ARW dynamics, with killing upon returning to $r_-$. 
\end{itemize}
Note that forcing particles to wake up during this procedure can only increase the probability that $u_n(r)>0$. Take $C_{\l}$ to be a positive constant such that 
if $\D = 2 \lceil C_\l  \log n \rceil $ and $G \sim Geometric (\l /(1+\l ))$ then 
	\[ \bbP (G \geq \D /2 ) \leq n^{-5} \]
for $n$ large enough. We mark all the points of the form 
	\[ \{ r_- + k\D : 1\leq k \leq K_n \}, \quad \mbox{ for }  
	K_n = \left\lfloor \frac{r-r_-}{\D} \right\rfloor -1. \] 
Note that $K_n>N$. We argue that on the event $\{ u_n (r_-) = N\}$ it is unlikely that $u_n(r) >0$, since even if we force the $N$ particles starting at $r_-$ to move until they occupy consecutive vertices of the $\Delta$-grid, upon release they would all fall asleep before being able to reach another grid site, and hence before any of them can reach $r$.

Restrict to the event  $\{ u_n (r_-) = N\}$. Release $N$ particles from $r_-$ according to (i)-(ii) above, except if a particle which is alive (and hence necessarily to the right of $r_-$) tries to fall asleep at a non-marked point, we force it to wake up, while if a particle reaches an unoccupied marked point we stop it there. At the end of this procedure we are left with at most $N$ active particles, at most one for each marked point. We let them evolve until they all settle. By the choice of $\D$, the probability that at least one of them travels by more than $\D/2$ steps before trying to fall asleep is at most $N n^{-5}$. But since the marked points are $\D$-spaced, if all the particles travel by at most $\D/2$ they cannot interact, so they all manage to fall asleep, and $u_n(r)=0$. Hence 
	\[ \begin{split} 
	\bbP \left(u_n(r) > 0  , \;  u_n(r_- ) \leq  \frac{n}{(\log n)^2} \right) & 
	\leq \sum_{N=1}^{\lceil n/(\log n)^2 \rceil } \bbP (u_n(r_-) =N , u_n(r)>0) 
	\\ & \leq\sum_{N=1}^{\lceil n/(\log n)^2 \rceil } N n^{-5} \leq n^{-2 }
	\end{split} \]
for $n$ large enough. 
\end{proof}

\begin{Remark}\label{rem2}
With the same arguments we could get an outer bound that holds with higher probability in $n$, at the price of changing the definition of $\z_{out}$. More precisely, for arbitrarily large $\g >0$ we can define 
	\[ \z_{out}(\g )  := \sup \{\z >0 : \bfE_\z  ( w(0)^{\g + 2 } )<\infty \}
	\]
to have that any $\e >0$
	\[ \bbP \left( 
	\mbox{supp}(u_n) \subseteq  B_{\frac{n}{\z_{out}(\g )} (1+\varepsilon )}  \right) 
	\geq 1 - n^{-\g } \]
for $n$ large enough. 
\end{Remark}

\subsection{Edge cases}
The main interest of Theorem~\ref{th:intro} is that it relates Experiment 2 to Experiments 1 and 3. But its applicability is limited by an incomplete understanding of Experiments 1 and 3. In particular, it would be interesting to rule out the edge cases $\z_{out}=0$ and $\z_{in}=1$.
By Definition~\ref{def:zout}, $\z_{out} \leq \z_c$ where $\z_c$ is the \emph{critical density}
	\[ \z_c := \sup \{\z >0 : \bfP_\z (w(0) < \infty)=1 \} \]
studied in \cite{rolla2019universality}. In fact, we conjecture that $\z_{out} = \z_c$.

It would be of interest to adapt the proof that $\z_c \geq \frac{\lambda}{1+\lambda}$ \cite[\textsection 4]{rolla2020} to give a lower bound on $\z_{out}$. 
That proof uses ``traps'' arising from a simple random walk in $\bbZ$ killed only on one side, so the resulting upper bound on $w(0)$ does not even have a first moment. To obtain an upper bound with finite moments, one could instead set traps on both sides of an interval $(-a,b)$ as in \cite{BGHR}, and then use the one-sided trap procedure to stabilize the half-lines $(-\infty,-a)$ and $(b,\infty)$. This procedure succeeds in stabilizing $\eta$ on $\bbZ$ provided that $w(-a)=w(b)=0$. 
The main difficulty is then to show that for small enough $\z$ the random variable 
	\[ b := \inf \{ x>0 \,:\, w(x)=0 \} \]
has finite moments.

This difficulty disappears in the case that all particles start at the origin, because there are no particles starting outside $(-a,b)$.  So a simple two-sided trap procedure in an interval $[-Cn,Cn]$ shows that there is a finite constant $C=C(\lambda)$ such that
	\begin{equation} \label{eq:linearouterbound} P(A_n \subset B_{Cn} \text{ eventually}) = 1. \end{equation}


Also of interest would be to adapt the proof in \cite{GHRR} that $\zeta_c<1$ to give an upper bound on $\z_{in}$.


\appendix

\section{IDLA results}\label{appendix}
In this section we collect shape theorems for standard IDLA on $\bbZ$, IDLA with killing upon exiting a given interval, and IDLA on Bernoulli percolation. 

\subsection{IDLA on $\bbZ$}
Let $A(0)=\{ 0\}$, and for $n\geq 1$  write $A_n$ for the IDLA cluster on $\bbZ$ obtained by adding $n$ particles to $A(0)$, all starting from the origin. Then 
	\[ A_n = A_{n-1} \cup \{ X_n \} , \]
where $X_n$ denotes the exit location of the $n^{th}$ random walk from $A_{n-1}$. Write 
	\[ A_n = [-a_n , b_n ] , \qquad n\geq 0\]
so that $a_0 = b_0 =0$ and  $a_n + b_n = n$ for all $n\geq 0$.  Then 
	\[ \bbE ( b_{n+1} | \cF_n ) = b_n + \bbP (X_{n+1} = b_n +1 ) 
	= b_n + 1-\frac{b_n+1}{n+2} = \left( \frac{n+1}{n+2} \right) (b_n +1) . \]
It follows that if 
	\[ M_n = (n+1) b_n - \frac{n(n+1)}2 , \qquad n\geq 0 \]
then the process $(M_n)_{n\geq 0}$ is a martingale with respect to the same filtration. 
Moreover, 
	\[ |M_n - M_{n-1} | = | n(b_n - b_{n-1} ) -(n- b_n) | \leq n \]
for all $n\geq 1$. 
\begin{Theorem}[Azuma-Hoeffding]
Let $(X_n )_{n\geq 0}$ be a martingale with respect to the filtration $(\cF_n )_{n\geq 0}$, and assume that $|X_k - X_{k-1}| \leq c_k $ almost surely, for all $k\geq 1$. Then for all $n\geq 1$ and all $\d >0$ 
	\[ \bbP ( |X_n - X_0 | > \d ) \leq 2 \exp \left( -\frac{\d^2}{2\sum_{k=1}^n c_k ^2} \right) . \]
\end{Theorem}
We apply the above inequality with $c_k = k$ and $\d = n^{3/2 +\e}$, to get 
	\[ \bbP (|M_n | > n^{3/2+\e} ) \leq 2 \exp \left( -\frac{n^{2\e}}{2} \right) 
	\leq e^{-n^{\e}} \]
where the last inequality holds for $n$ large enough. Since on the high probability event 
$\{ |M_n| \leq n^{3/2+\e} \}$ we have that 
	\[ \left| \frac{b_n}n -\frac 12 \right| \leq n^{-1/2+\e} , \qquad 
	 \left| \frac{a_n}n -\frac 12 \right| \leq n^{-1/2+\e} , \]
this proves the following shape theorem for IDLA on $\bbZ$. 
\begin{Theorem}\label{th:IDLAshape}
Let $(A_n)_{n\geq 0}$ denote an IDLA process on $\bbZ$ starting from $A(0)=\{0\}$, and write $B_r$ for the discrete ball of radius $r$ on $\bbZ$ centred at $0$. Then for all $\e >0$ it holds 
	\begin{equation}\label{IDLAshape}
	 \bbP \Big( B_{n/2 - n^{1/2+\e} } \subseteq A_n \subseteq B_{n/2 + n^{1/2+\e} } 
	\Big) \geq 1-e^{-n^\e } , 
	\end{equation}
for $n$ large enough. 
\end{Theorem}

\subsection{IDLA with killing upon exiting $I$} \label{sec:appkilling}
We now use the same martingale to prove an analogous result for IDLA stopped upon exiting a given interval. Fix $\d >0$ and let 
	\[ I =\left[ -\frac n2 (1-\d ) , \frac n2 (1-\d ) \right] . \]
We write $(\bar A_n)_{n\geq 0}$ for the IDLA process on $\bbZ$ starting from $\bar A(0) =\{0\}$ with particles killed upon exiting the interval $I$. If $\bar A_n = [-\bar a_n , \bar b_n ]$, let
	\[ T = \inf \Big\{ n\geq 0 : \bar b_n = \frac n2 (1-\d ) \mbox{ or } \bar a_n = \frac n2 (1-\d ) \Big\} \]
denote the first time the cluster reaches the boundary of $I$. Then up to time $T$ we can couple the process $(\bar A_n)_{n\geq 0}$ with a standard IDLA process $(A_n)_{n\geq 0}$ so that $\bar A_n = A_n $ for all $n\leq T$. 

Fix $\e \in (0,1/2)$ and suppose that $n$ is large enough so that \eqref{IDLAshape} holds. Then with probability exceeding $1- e^{-n^\e}$ we have that for all $k\leq n(1-\d ) (1- n^{-1/2+2\e} )$ 
	\[  b_k \leq b_{n(1-\d ) (1- n^{-1/2+2\e} )} 
	\leq \frac{n(1-\d )}2 (1- n^{-1/2+2\e} ) + n^{1/2+\e } \leq \frac{n(1-\d)}2 - n^{1/2+\e} , \]
which implies $T> n(1-\d ) (1- n^{-1/2+2\e} )$. 

Assume that $b_T = \frac n2 (1-\d )$. Then on the high probability event $T> n(1-\d ) (1- n^{-1/2+2\e} )$ 
	\[ a_T = T-b_T \geq \frac{n}2 (1-\d ) ( 1-2n^{-1/2+2\e }) , \]
so that $ \# (I \setminus A_T ) \leq n^{1/2+2\e}$. Moreover, since necessarily $T \leq n(1-\d)$, at time $T$ there are still at least $n\d $ particles to be released, with killing upon exiting $I$. In order for $\bar A_n = I$ it thus suffices that at least $n^{1/2+2\e}$ of them settle before reaching the right boundary of $I$. Since out of $n\d $ random walks the number of walks that reach $-\frac n2 (1-\d ) $ before $\frac n2 (1-\d )$ is a Binomial random variable $B(n\d , 1/2)$, we find 
	\[ \bbP (B(n\d , 1/2 ) < n^{1/2+2\e} ) \leq \bbE ( e^{-B(n\d , 1/2)} ) e^{n^{1/2+2\e }} 
	\leq \left( \frac{e^{-1}+1}{2} \right)^{n\d } e^{n^{1/2+2\e }}  \leq e^{-n \d /10 } \]
for $n$ large enough. In all, 
	\[ \bbP (\bar A_n \subsetneq I ) \leq \bbP (T\leq n(1-\d ) (1- n^{-1/2+2\e} ) ) 
	+ \bbP (B(n\d , 1/2 )< n^{1/2+2\e} )
	\leq 2 e^{-n^\e } \]
for $n$ large enough. This proves Lemma \ref{le:IDLA}. 

\subsection{IDLA on Bernoulli percolation}
Let $(\o_z)_{z\in \bbZ}$ denote a Bernoulli$(p)$ vertex percolation configuration, in which each vertex $z \in \bbZ$ is open $(\o_z =1)$ with probability $p\in (0,1)$, and closed $(\o_z =0)$ with probability $1-p$. We let $(A_n^p)_{n\geq 0}$ denote an IDLA process starting from $A^p_0 = \{0\}$, with particles only allowed to settle at open sites of the percolation configuration, and refer to such a process as $p$-IDLA. Denote by $\bfP_p$ the distribution of the $p$-IDLA process, averaged on the environment. We further let $v_n$ denote the odometer function associated to the stabilization of the initial particle configuration $\eta = n \d_0$ consisting of $n$ active particles at the origin. Note that here all the stack instructions are movement instructions, so that $v_n (x)$ counts the number of particle emissions from $x$ until stabilization. 
For $p$-IDLA we have the following shape theorem. 
\begin{Theorem}\label{th:IDLABernoulli}
For all $\e >0$ and $n$ large enough, 
	\[ \bfP_p \left( B_{\frac{n}{2p}(1-\e)} \subseteq supp (v_n) \subseteq 
	B_{\frac{n}{2p}(1+\e)} \right) \geq 1-e^{-n / (\log n)^2} . \]
\end{Theorem}
A similar result appears in \cite{ben2000asymptotic}, Theorem 4, where the focus is on higher dimensions. Since in dimension $1$ the proof is substantially simpler, and we need quantitative estimates that do not appear in \cite{ben2000asymptotic}, we include it here for completeness. 

\subsubsection*{Proof of the inner bound}
Fix $\e >0$ as in the statement. 
Let us say that a particle \emph{settles} when it first reaches an open empty site. 
Following Lawler, Bramson and Griffeath \cite{lawler1992internal}, we think of the random walks' trajectories as continuing indefinitely, even after the corresponding particles settle. Then for arbitrary $z \in B_{\frac{n}{2p}(1-\e)}$ define 
	\[ \begin{split} 
	N_z & = \sum_{k=1}^n \one \big\{ 
	 \mbox{The } k^{th} \mbox{ walk reaches }z \mbox{ before settling} \big\} 
	\\ M_z & = \sum_{k=1}^n \one \big\{  \mbox{The } k^{th} \mbox{ walk reaches } z \mbox{ before exiting }B_{n/2p} \big\} \\
	 L_z & = \sum_{k=1}^n \one \big\{ \mbox{The } k^{th} \mbox{ walk settles before reaching }z, \mbox{ and reaches } z \mbox{ before exiting }B_{n/2p} \big\} . 
	\end{split}\]
Then $N_z \geq M_z - L_z$, and so 
	\begin{equation}\label{eqA}
	 \bfP_p \left( z \notin  supp (v_n) \right) 
	= \bfP_p (N_z =0 ) \leq \bfP_p (M_z \leq a ) + \bfP_p (L_z \geq a ) 
	\end{equation}
for $a \in \bbR$ to be chosen later. 

Let $\bbP_x$ denote the law of a simple random walk on $\bbZ$ starting from $x$, and write $\tau_z$ and $\tau_{B_{n/2p}}$ for the hitting time of $z$ and the exit time from the ball $B_{n/2p}$ respectively. Then 
	\[ \bfE_p (M_z) = np \bbP_0 (\tau_z < \tau_{B_{n/2p}} ) 
	= np \bigg(  \frac{\frac n{2p}}{\frac n{2p} + |z|} \bigg) \geq \frac{np}{2-\e} \]
where the last inequality holds for $n$ large enough. 
Moreover, if
	\[  \tilde{L}_z = \sum_{y\in B_{n/2p}} \o(y) \one_y (\tau_z < \tau_{B_{n/2p}} ) , \]
then $L_z \leq \tilde{L}_z$, and 
	\[ \bfE_p (L_z) \leq \bfE_p (\tilde L_z) 
	= p \sum_{y \in B_{n/2p}} \bbP_y ( \tau_z < \tau_{B_{n/2p}}) . 
	\]
Let $f$ denote the piecewise linear function whose graph in $[-1/(2p) , 1/(2p)]$ consists of the two line segments connecting $(-1/(2p) , 0)$ with $(|z|/n , 1)$ and $(|z|/n ,1)$ with $(1/(2p) , 0)$. 
Using the integral approximation 
	\[ \bbP_y ( \tau_z < \tau_{B_{n/2p}}) = \int_{-\frac 1{2p}}^{\frac 1{2p}} f(x) dx  + o (1) \]
as $n\to\infty$, we find 
	\[ \lim_{n\to \infty } \frac{ 1}{n} \bbP_y ( \tau_z < \tau_{B_{n/2p}}) 
	= \frac 1 2 , 
	\]
from which 
	\[ \frac{np}{2} (1-\e ) \leq 
	\bfE_p (\tilde L_z)
	\leq \frac{np}{2} \left( 1+ \frac \e 8 \right)  \]
for $n$ large enough. 
Then using a standard concentration inequality  for sums of indicators  (see \cite{alon2004probabilistic}, Corollary A.1.14), we obtain that there exists $c_\e >0$ such that 
	\[\begin{split} 
	 \bfP_p \left( M_z \leq  \bfE_p (\tilde{L}_z )\left( 1+\frac \e 4 \right) \right) & \leq  
	\bfP_p \left( M_z \leq  \bfE_p (M_z ) \left( 1-\frac \e {16} \right) \right) 
	\\ & \leq 2 \exp \left\{ - c_\e \bfE_p (M_z)\right\} 
	\leq \exp \left\{ -\frac{ n}{\log n} \right\} 
	\end{split} \]
and 
	\[\begin{split} 
	 \bfP_p \left( \tilde L_z \geq  \bfE_p (\tilde{L}_z )\left( 1+\frac \e 4 \right) \right) 
	 & \leq 2 \exp \left\{ - c_\e \bfE_p (\tilde L_z) \right\} 
	\leq \exp \left\{ -\frac{ n}{\log n} \right\} 
	\end{split} \]
for $n$ large enough. 
This shows that taking  $a =  \bfE_p (\tilde{L}_z) \left( 1+\frac \e 4 \right) $ in \eqref{eqA} yields 
	\[  \bfP_p \left( B_{\frac{n}{2p}(1-\e)} \nsubseteq supp (v_n)  \right)  
	 \leq \frac np \exp \left\{ -\frac{ n}{\log n} \right\} \leq e^{- n / (\log n)^2 }\]
for $n$ large enough, which proves the inner bound.

\subsubsection*{Proof of the outer bound}
Fix $\e >0$ as in the statement. Then for $\l >0$ small enough 
	\[ \begin{split} 
	\bfP_p \bigg( \sum_{x \in B_{\frac n{2p} (1-\frac \e 4 )}} \o (x)&  \leq n \Big( 1-\frac \e 4 \Big)^2 \bigg)
	 \leq \bfE_p \bigg( \exp \bigg\{ -\l  \sum_{x \in B_{\frac n{2p} (1-\frac \e 4 )}} \o (x) \bigg\}  \bigg) e^{\l n (1-\e /4)^2 } 
	\\ & = ( 1-p + pe^{-\l } )^{\frac n p (1-\e /4 ) } e^{\l n (1-\e /4)^2 } 
	\leq e^{-\l n (1-\e /4) \e /4 } \leq e^{-\frac n{\log n}} 
	\end{split}\]
for $n$ large enough. It follows that 
	\begin{equation}
	 \begin{split}
	\bfP_p \bigg( &\Big| A_n^p  \setminus B_{\frac n{2p} (1-\frac \e 4 )}  \Big| > \frac{n\e }{2} \bigg)\leq  \\ & \leq \bfP_p \bigg( \sum_{x \in B_{\frac n{2p} (1-\frac \e 4 )}} \o (x) \leq n \Big( 1-\frac \e 4 \Big)^2 \bigg) + \bfP_p \left( B_{\frac{n}{2p}(1-\e /4)} \nsubseteq supp (v_n)  \right)  
	 \leq 2 e^{-\frac n{\log n}} 
	\end{split} 
	\end{equation}
for $n$ large enough. On the high probability event 
	\begin{equation} \label{eq:outer}
	 \big| A_n^p  \setminus B_{\frac n{2p} (1-\frac \e 4 )} \big| \leq \frac{n\e}{2} 
	 \end{equation}
 at most $n\e /2$ particles settle outside the ball of radius $\frac{n}{2p} (1-\frac \e 4)$. We now show that, even if they all settle on one side, the outer bound is satisfied with high probability. To this end, let $(\D_k)_{k \in \bbZ \setminus \{ 0\}}$ denote the gaps between consecutive open sites to the right of $\lfloor  \frac{n}{2p}(1- \frac \e 4)\rfloor$ if $k>0$, and to the left of $-\lfloor  \frac{n}{2p}(1-\frac \e 4)\rfloor$ if $k<0$. 
\begin{figure}[!h] \label{fig:IDLABproof}
  \centering
    \includegraphics[width=.9\textwidth]{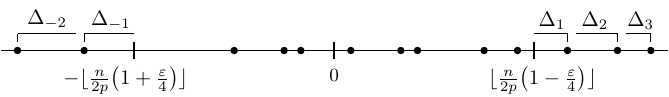}
    \caption{An illustration of the $\Delta_k$ notation. }
\end{figure}

Then on the high probability event  \eqref{eq:outer} it must be 
	\[ supp(v_n) \subseteq B_{ \frac n{2p} ( 1-\frac \e 4) + \max \big\{ \sum_{k=1}^{n\e /2} \D_k , \, \sum_{k=1}^{n\e /2} \D_{-k} \big\} } . \]
It thus suffices to show that 
	\[ \bfP_p \bigg( \max \bigg\{ \sum_{k=1}^{n\e /2} \D_k , \,  \sum_{k=1}^{n\e /2} \D_{-k} \bigg\} \geq \frac {n\e}{2p}\Big( 1+ \frac \e  4 \Big) \bigg) \leq 2 e^{-\frac n{\log n}} \]
for $n$ large enough. To this end, let $(G_k)_{k\in \bbZ \setminus \{ 0\}}$ be a sequence of i.i.d.\ Geometric$(p)$ random variables, and note that since $\Delta_1 \preceq G_1$, $\Delta_2 \preceq G_2$ we have the stochastic domination
	\[  \sum_{k=1}^{n\e /2} \D_k \preceq  \sum_{k=1}^{n\e /2} G_k , \qquad 
	 \sum_{k=1}^{n\e /2} \D_{-k} \preceq  \sum_{k=1}^{n\e /2} G_{-k} . \]
Thus, choosing $\l >0 $ small enough depending on $p$ and $\e$, we find 
	\[ \begin{split} 
	\bbP \left(  \sum_{k=1}^{n\e /2} G_k \geq \frac{n\e}{2p} ( 1+\e /4) \right) 
	& \leq \big[ \bbE (e^{\l G_1} )\big]^{n\e /2} \exp \Big\{ - \frac{n\e}{2p} ( 1+\e /4)\Big\} 
	\\ & = \bigg( \frac{pe^\l}{1-(1-p)e^\l} \bigg)^{n\e/2} \exp \Big\{ - \frac{n\e}{2p} ( 1+\e /4)\Big\} 
	\\ & \leq \exp \Big\{ - \frac{n\e}{2p} ( 1+\e /8)\Big\} \leq e^{-\frac n{\log n}} 
	\end{split} \]
for $n$ large enough, which concludes the proof.

\smallskip 

\bibliography{ARWbib}
\bibliographystyle{plain}

\end{document}